\documentclass[a4paper,10pt,twocolumn,preprint,3p]{elsarticle}
\usepackage[utf8]{inputenc}
\usepackage[T1]{fontenc}
\usepackage[english]{babel}
\usepackage{amsmath}
\usepackage{amssymb}
\usepackage{amsthm}
\usepackage{graphicx}
\usepackage{xcolor}
\usepackage{lmodern}

\theoremstyle{plain}
\newtheorem{thm}{Theorem}

\theoremstyle{definition}
\newtheorem{dfn}{Definition}

\theoremstyle{remark}
\newtheorem{rmk}{Remark}

\renewcommand{\vec}{\mathbf}

\newcommand{\IR}{\mathbb{R}}
\newcommand{\IN}{\mathbb{N}}
\newcommand{\bmat}[1]{\begin{bmatrix} #1 \end{bmatrix}}

\begin{document}

\begin{frontmatter}
	\title{Discrete-time port-Hamiltonian systems:\\ A definition based on symplectic integration
	}

	\author[1]{Paul Kotyczka\corref{cor1}
	} 
		\ead{kotyczka@tum.de}
	\author[2]{Laurent Lefèvre
	} 
		\ead{laurent.lefevre@lcis.grenoble-inp.fr}
	
	\cortext[cor1]{Corresponding author}
	\address[1]{Technical University of Munich, Department of Mechanical Engineering, Chair of Automatic Control, Boltzmannstra{\ss}e 15, 85748 Garching, Germany}
	\address[2]{Univ. Grenoble Alpes, LCIS, 50 rue Barth{\'e}l{\'e}my de Laff{\'e}mas, 26902 Valence, France}

	\begin{abstract}                
		We introduce a new definition of discrete-time port-Hamiltonian systems (PHS), which results from structure-preserving discretization of explicit PHS in time. We discretize the underlying continuous-time Dirac structure with the collocation method and add discrete-time dynamics by the use of symplectic numerical integration schemes. The conservation of a discrete-time energy balance -- expressed in terms of the discrete-time Dirac structure -- extends the notion of symplecticity of geometric integration schemes to open systems. We discuss the energy approximation errors in the context of the presented definition and show that their order is consistent with the order of the numerical integration scheme. Implicit Gauss-Legendre methods and Lobatto IIIA/IIIB pairs for partitioned systems are examples for integration schemes that are covered by our definition. The statements on the numerical energy errors are illustrated by elementary numerical experiments.
	\end{abstract}
	\begin{keyword}
	    Port-Hamiltonian systems, Dirac structures, discrete-time systems, geometric numerical integration, symplectic methods.
	\end{keyword}
\end{frontmatter}

\section{Introduction}
\label{sec:04:010}

The \emph{geometric} integration of ordinary differential equations, see e.\,g. \cite{leimkuhler2004simulating}, \cite{hairer2006geometric}, is an important approach to perform long-time simulations of Hamiltonian systems. \emph{Symplectic} integration conserves not only the symplectic form in the (mechanical) phase space, but also invariants of motion (Casimirs, first integrals). Symplectic integrators that are derived based on discrete versions of Hamilton's principle are called \emph{variational integrators}. They ``\emph{work very well for both conservative and dissipative or forced mechanical systems}''\footnote{See \cite{lew2004overview}, Section 2.}. 

The port-Hamiltonian (PH) approach, see \cite{duindam2009modeling} for an overview, is very appealing for modeling, simulation and control of complex multiphysics systems. PH systems generalize the \emph{Hamiltonian} system representation by the additional definition of \emph{ports}, which are pairs of power variables that characterize energy exchange in the system and over its boundary. The numerical integration of PH systems has to account for this energy exchange. Besides the error of the numerical solution itself, the error of the energy transmitted over the \emph{discrete ports} is of fundamental interest in the simulation of PH systems. 

Most existing works on the discrete-time formulation of PH systems make use of a \emph{discrete gradient}, defined from a finite-differences point of view (\cite{goren2008gradient}, \cite{aoues2013canonical}, \cite{falaize2016passive}). A generic definition of PH dynamics on discrete manifolds (spaces that locally look like discretization grids or the set of floating-point numbers) is given in  \cite{talasila2006discrete}. Objects and operations from differential geometry are adapted to the discrete setting and discrete-time Dirac structures are defined. In the discrete setting, the chain rule is not valid, which means that the change of energy over a sampling interval is only \emph{approximated} by a product of the discrete gradient $\Game_x H(x)$ and the increment $\Delta x$ of the state.

We give a new definition of discrete-time Dirac structures and discrete-time PHS, which is based on the approximation of the continuous-time structural energy balance and  symplectic numerical time integration by \emph{collocation methods}. The corresponding quadrature formulas allow for quantitative statements on the approximation error both of the solution and the supplied energy. We show that only Gauss-Legendre collocation, applied to linear PHS, guarantees an \emph{exact} discrete energy balance as defined in \cite{celledoni2017energy}, Def. III.2. Our definition includes discretization schemes, which yield a non-exact discrete energy balance. An example are the Lobatto IIIA/IIIB methods for partitioned systems. The energy error is then \emph{consistent with}, i.\,e. it has the same order as the chosen integration scheme.

The paper is organized as follows. In Section \ref{sec:04:020}, we introduce the considered class of finite-dimensional PH systems and the underlying Dirac structure.  Section \ref{sec:04:030} contains as main results the definitions of discrete-time Dirac structures and PH systems based on the collocation method. In Section \ref{sec:04:040}, we consider Gauss-Legendre methods and Lobatto IIIA/IIIB pairs and discuss the order of the energy approximations. Section \ref{sec:04:050} illustrates the statements on the elementary example of a linear undamped/damped oscillator. In the concluding Section \ref{sec:04:060}, we sum up the paper, and we point out perspectives for future work based on the presented results.

This paper is inspired by early results on the symplectic time integration of PH systems using Gauss-Legendre collocation \cite{kotyczka2018discrete}. The main novelties are the precise consideration of the different energy approximations, the application of the ideas to more general schemes including $s$-stage Lobatto pairs for partitioned systems, the analysis and order proofs for the energy errors, and the extended section on numerical experiments.

\section{Finite-dimensional PH systems}
\label{sec:04:020}
We consider the class of \emph{lossless} finite-dimensional PH systems in an \emph{explicit} input-state-output representation (see e.\,g. \cite{duindam2009modeling} or \cite{schaft2014port})
\begin{subequations}
\label{eq:04:0010}
\begin{align}
	\label{eq:04:0010a}
    \dot{\vec x}(t) &= \vec J(\vec x(t)) \nabla H(\vec x(t)) + \vec G(\vec x(t)) \vec u(t)\\
    \label{eq:04:0010b}
    \vec y(t) &= \vec G^T(\vec x(t)) \nabla H(\vec x(t))
\end{align}
\end{subequations}
with state vector $\vec x \in \IR^n$, collocated in- and output vectors $\vec u, \vec y \in \IR^m$. The Hamiltonian $H: \IR^n \rightarrow \IR$ is bounded from below with a strict minimum in $\vec x^*$, which is the equilibrium state for $\vec u \equiv \vec 0$. By skew-symmetry of the interconnection matrix $\vec J = - \vec J^T$ and the definition of the collocated output, the differential energy balance
\begin{equation}
    \label{eq:04:0020}
     \dot H(\vec x(t)) = \vec y^T(t) \vec u(t)
\end{equation} 
holds, or in integral form,
\begin{equation}
    \label{eq:04:0030}
    H(\vec x(t_2)) - H(\vec x(t_1)) = \int_{t_1}^{t_2}\!\!\vec y^T(s) \vec u(s) \, ds, \quad\! \forall t_1 \!\leq\! t_2,
\end{equation}
which shows passivity (see e.\,g. \cite{schaft2017l2}) of the state representation \eqref{eq:04:0010}. The
energy balance is a \emph{structural} or \emph{geometric} property, i.\,e. it holds independently of $H(\vec x)$. \emph{Flow} and \emph{effort} vectors are defined as
\begin{equation}
    \label{eq:04:0040}
    \vec f(t) := - \dot{\vec x}(t), \qquad \vec e(t) := \nabla H(\vec x(t)).
\end{equation}
Because of $\dot H = (\nabla H)^T \dot{\vec x} = - \vec e^T \vec f$, they represent \emph{power-conjugated}, \emph{dual} variables. The differential energy balance \eqref{eq:04:0020} can be written as the power balance equation on the \emph{bond space} $\mathcal F \times \mathcal E$, with $\mathcal F = \IR^n \times \IR^m \ni (\vec f, \vec y)$ and $\mathcal E = \IR^n \times \IR^m \ni (\vec e, \vec u)$:
\begin{equation}
    \label{eq:04:0050}    
    \vec e^T(t) \vec f(t) + \vec y^T(t) \vec u(t) = 0.
\end{equation}
By 
\begin{subequations}
\label{eq:04:0060}
\begin{align}
	\label{eq:04:0060a}
    - \vec f(t) &= \vec J(\vec x(t)) \vec e(t) + \vec G(\vec x(t)) \vec u(t)\\
    \label{eq:04:0060b}
    \vec y(t) &= \vec G^T(\vec x(t)) \vec e(t),
\end{align}
\end{subequations}
the bond variables are constrained to a subspace (i.\,e. the graph of the skew-symmetric map defined in \eqref{eq:04:0060}), on which in particular \eqref{eq:04:0050} holds. This subspace is called a \emph{Dirac structure}. For more details on Dirac structures and PH systems, see e.\,g. \cite{schaft2017l2}, Chapter 6.

\section{Discrete-time PH systems based on collocation}
\label{sec:04:030}
We define the class of \emph{discrete-time PH systems}, which arise from a \emph{discrete-time Dirac structure}. The latter is obtained by applying the collocation method to the class of PH systems \eqref{eq:04:0010} and by defining in an appropriate manner \emph{discrete flow and effort vectors} for every sampling interval. Special attention is paid on the discretization of the energy balance \eqref{eq:04:0030}. For a \emph{consistent} discrete-time approximation of the PH system \eqref{eq:04:0010}, both the \emph{supplied energy} (right hand side of \eqref{eq:04:0030}) and the \emph{stored energy} (left hand side of \eqref{eq:04:0030}) must be approximated with the same order.

\subsection{Collocation method}
\label{subsec:04:030:010}
We consider equidistant sampling intervals $I^k = [t^{k}_0, t^{k}_{s+1}]=[(k-1)h,kh]$, $k \in \IN$ for the time $t$ with $t^k_{s+1} = t^k_0 + h$, see Fig. \ref{fig:04:0010}. With $t=((k-1)+\tau)h$, the sampling intervals are parametrized by the normalized time $\tau \in [0,1]$. The polynomial approximations of the system variables will be denoted with a tilde. As described in Section II.1.2 of \cite{hairer2006geometric}, the numerical approximation of the solution $\vec x(t)$ of \eqref{eq:04:0010} is given by the vector $\tilde{\vec x}(t)\in \IR^n$ of \emph{collocation polynomials} of degree $s$. Assume first the initial value $\vec x_0^k:= \tilde{\vec x}(t_0^k) = \vec x(t_0^k)$ to be known. The continuous numerical solution $\tilde{\vec x}(t)$ is then the vector of polynomials whose time derivative matches the right hand side of \eqref{eq:04:0010a} in the $s$ collocation points $t_i^k := t_0^k + c_i h$ with $0 \leq c_i \leq 1$:
\begin{align}
	\label{eq:04:0070}
        \dot{\tilde{\vec x}}(t_i^k) &= - \vec f_i^k, \hspace{2cm} i=1,\ldots s,\\
    \nonumber
        - \vec f_i^k &= \left. \left ( \vec J(\vec x) \nabla H(\vec x) \right) \right|_{\vec x = \tilde{\vec x}(t_i^k)} + \vec G(\tilde{\vec x}(t_i^k)) \vec u(t_i^k).
\end{align}

\emph{Notation:} Arguments in \emph{latin} letters ($t$ or $s$ under the integral), refer to time functions evaluated on $I^k$. \emph{Greek} letters ($\tau$ or $\sigma$) refer to the same function, mapped to the normalized interval $[0,1]$.

\begin{figure}
    \centering
    \includegraphics[scale=0.97]{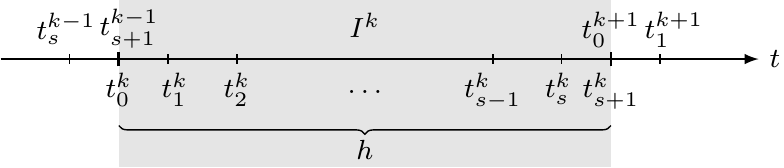}
    \caption{Sampling interval $I^k$ with interior collocation points $t_i^k = t_0^k + c_i h$, $i=1,\ldots,s$.}
    \label{fig:04:0010}
\end{figure}

\subsection{Approximation of flow and state variables}
\label{subsec:04:030:020}
Based on $\vec f_i^k \in \IR^n$, $i=1,\ldots, s$, according to \eqref{eq:04:0070}, the interpolation formula 
\begin{equation}
    \label{eq:04:0080}
    \dot{\tilde{\vec x}}(t_0^k + \tau h) =: - \tilde{\vec f}(t_0^k + \tau h) = - \sum_{i=1}^s \vec f_i^k \ell_i(\tau),
\end{equation}
with $\ell_i$ the $i$-th Lagrange interpolation polynomial
\begin{equation}
    \label{eq:04:0090}
    \ell_i(\tau) =  \prod_{\substack{j=1\\j\neq i}}^{s} \frac{\tau - c_j}{c_i - c_j},  \qquad \tau \in [0,1],
\end{equation}
gives a polynomial approximation of $\dot{\vec x}(t)$ on $I^k$. The flow coordinates are merged in the \emph{discrete-time flow vector}
\begin{equation}
    \label{eq:04:0100}
    \vec f^k := \bmat{ (\vec f_1^k)^T & \ldots & (\vec f_s^k)^T}^T \in \IR^{sn}, 
\end{equation}
based on which the numerical solution $\tilde{\vec x}(t^k + \tau h)$, $\tau \in [0,1]$ is obtained by integration of \eqref{eq:04:0080}:
\begin{equation}
	\label{eq:04:0110}
    \tilde{\vec x}(t_0^k + \tau h) 
    = \tilde{\vec x}(t_0^k) - h \sum_{j=1}^s \Big( \vec f_j^k \int_0^\tau \ell_j(\sigma) \, d\sigma \Big).
\end{equation}
The values $\vec x_i^k := \tilde{\vec x}(t_i^k)$ of the numerical solution inside and at the end of the interval $I^k$ are then computed as
\begin{align}
	\label{eq:04:0120}
    \vec x_i^k &= \vec x_0^k - h \sum_{j=1}^s a_{ij} \vec f_j^k, \quad i=1,\ldots, s,\\
    \label{eq:04:0130}
    \vec x_{s+1}^k &= \vec x_0^k - h \sum_{j=1}^s b_{j} \vec f_j^k, 
\end{align}
with\footnote{These values are, together with $c_i$, the coefficients of the Butcher table for the Runge-Kutta (RK) interpretation of the collocation method, see \cite{hairer2006geometric}, Theorem II.1.4.} $(i,j = 1,\ldots, s)$
\begin{equation}
    \label{eq:04:0140}
    a_{ij} = \int_0^{c_i} \ell_j(\sigma)\, d\sigma, \quad  b_j = \int_0^1 \ell_j(\sigma)\, d\sigma.
\end{equation}  
In \emph{continuous} collocation methods, the numerical solution at the start $t_0^{k+1} = t_{s+1}^k$ of the subsequent interval is initialized by $\vec x_0^{k+1} = \vec x_{s+1}^k$.

\subsection{Effort approximation and discrete structure equation}
\label{subsec:04:030:030}
The definition of the discrete flow coordinates $\vec f_1^k, \ldots, \vec f_s^k$ in \eqref{eq:04:0070} requires to evaluate the effort vector $\nabla H(\vec x(t))$, the input $\vec u(t)$ and the interconnection and input matrices $\vec J(\vec x(t))$, $\vec G(\vec x(t))$ \emph{in the flow collocation points} $c_i$,
\begin{equation}
    \label{eq:04:0150}
    \vec e_i^k :=  \left. \nabla H(\vec x) \right|_{\vec x_i^k}, \quad 
    \vec u_i^k := \vec u(t_0^k + c_i h),
\end{equation}
and
\begin{equation}
	\label{eq:04:0160}
    \vec J_i^k := \vec J(\vec x_i^k), \quad \vec G_i^k := \vec G(\vec x_i^k)
\end{equation}
for $i=1,\ldots,s$. The discrete-time counterpart of \eqref{eq:04:0060a} is then
\begin{equation}
    \label{eq:04:0170}
    - \vec f_i^k = \vec J_i^k \vec e_i^k + \vec G_i^k \vec u_i^k, 
\end{equation}
with $\vec J_i^k = - (\vec J_i^k)^T$. With $\vec e_i^k \in \IR^n$ and $\vec u_i^k \in \IR^m$  the \emph{discrete effort} and \emph{discrete input coordinates} according to \eqref{eq:04:0150}, the polynomial approximations of the effort and the input vector are
\begin{equation}
    \label{eq:04:0180}
    \tilde{\vec e}(t_0^k + \tau h) = \sum_{i=1}^{s} \vec e_i^k \ell_i(\tau), \quad\!
    \tilde{\vec u}(t_0^k + \tau h) =\ \sum_{i=1}^{s} \vec u_i^k \ell_i(\tau).
\end{equation}
In accordance with the approximate flows, we define the \emph{discrete-time effort vector} and the \emph{discrete-time input vector}
\begin{align}
    \label{eq:04:0190}
    \vec e^k &= \bmat{ (\vec e_1^k)^T & \ldots & (\vec e_s^k)^T}^T \in \IR^{sn},\\
    \label{eq:04:0200}
    \vec u^k &= \bmat{(\vec u_1^k)^T & \ldots & (\vec u_s^k)^T}^T \in \IR^{sm}.
\end{align}
Defining the block-diagonal matrices
\begin{equation}
	\label{eq:04:0210}
    \begin{split}
    \vec  J^k = - (\vec J^k)^T &= \mathrm{blockdiag}( \vec J_1^k, \ldots, \vec J_s^k), \\
    \vec G^k &= \mathrm{blockdiag}( \vec G_1^k, \ldots, \vec G_s^k),
    \end{split}
\end{equation}
the \emph{structure equation} \eqref{eq:04:0170} on the sampling interval $I^k$ can be rewritten as
\begin{equation}
    \label{eq:04:0220}
    - \vec f^k = \vec J^k \vec e^k + \vec G^k \vec u^k.
\end{equation}

\begin{rmk}
    Defining the discrete effort and input coordinates based on \emph{different} collocation points $d_1,\ldots,d_r$ is also conceivable. In this case, the terms of the right hand side of \eqref{eq:04:0170} must be replaced by interpolations between the effort collocation points according to \eqref{eq:04:0180}. In this paper, we restrict ourselves to \emph{identical} collocation points for the flow and effort variables, and to the explicit representation of the resulting Dirac structure and PH system.
\end{rmk}

\subsection{Discrete-time supplied energy}
\label{subsec:04:030:040}
{Since the instantaneous (local) power balance results trivially from the equations of the Dirac structure, we seek to express a discrete-time counterpart of the integral energy balance equation \eqref{eq:04:0030} on the time interval $I^k$. To this end, we integrate the polynomial approximation of instantaneous power $- \tilde{\vec e}^T(t_0^k + \tau h) \tilde{\vec f}(t_0^k + \tau h)$ over the normalized time interval $[0,1]$, and obtain an \emph{approximation of supplied energy} on the sampling interval $I^k$
\begin{equation}
\label{eq:04:0240}
    \Delta \tilde H^k :=  - h \int_{0}^{1} \tilde{\vec e}^T(t_0^k + \sigma h) \tilde{\vec f}(t_0^k + \sigma h) \, d\sigma.
\end{equation}
Substituting the definitions \eqref{eq:04:0080} and \eqref{eq:04:0180} of the polynomial flow and effort approximations, we obtain the bilinear form
\begin{equation}
    \label{eq:04:0250}
    \Delta \tilde H^k = - h (\vec e^k)^T \vec M \vec f^k
\end{equation}
with the symmetric matrix $\vec M = \vec M^T \in \IR^{sn\times sn}$, 
\begin{equation}
    \label{eq:04:0260}
    \vec M \!=\! \bmat{m_{11} \!&\! \ldots \!&\! m_{1s}\\ \vdots \!&\! \ddots \!&\! \vdots \\ m_{s1} \!&\! \ldots \!&\! m_{ss}} \!\otimes \vec I_n,
    \;\; m_{ij} \!=\!\! \int_0^1\!\! \ell_i(\sigma) \ell_j(\sigma) d\sigma,
\end{equation}
where $m_{ij} = m_{ji}$ and $\otimes$ denotes the Kronecker product.

\begin{rmk}
	We can understand the term $\vec M \vec e^k$ as a generalization of the \emph{discrete gradient} and $- h \vec f^k$ as a vector generalizing the increment of the numerical solution in the integration step.
\end{rmk}

\subsection{Discrete-time Dirac structure}
\label{subsec:04:030:050}

We provide conditions under which the polynomial approximation of the power variables leads to the definition of a discrete-time Dirac structure. Substituting the relation \eqref{eq:04:0220} in the right hand side of the discrete energy balance  \eqref{eq:04:0250}, we obtain
\begin{equation}
    \label{eq:04:0270}
    - h (\vec e^k)^T \vec M \vec f^k = h (\vec e^k)^T \vec M \vec J^k \vec e^k + h (\vec e^k)^T \vec M \vec G^k \vec u^k.
\end{equation}
At this stage, we want to recover a discrete-time equivalent of the structural power balance \eqref{eq:04:0050}. To this end, the first term on the right hand side must vanish: $h (\vec e^k)^T \vec M \vec J^k \vec e^k \overset{!}{=} 0$ for all $\vec e^k \in \IR^{sn}$, or written element-wise (recall $m_{ij}=m_{ji}$),
\begin{equation}
	\label{eq:04:0280}
    h \sum_{i=1}^s \sum_{j=1}^s (\vec e_i^k)^T m_{ij} \vec J_j^k \vec e_j^k  \overset{!}{=} 0.
\end{equation}
By skew-symmetry of $\vec J_j^k$, we have $(\vec e_j^k)^T m_{jj} \vec J_j^k \vec e_j^k = 0$ for all $j=1,\ldots,s$. The remaining elements of the sum cancel to zero, if $ (\vec e_i^k)^T m_{ij} \vec J_j^k \vec e_j^k = - (\vec e_j^k)^T m_{ji} \vec J_i^k \vec e_i^k$ holds for all $i \neq j$. With $(\vec e_i^k)^T m_{ij} \vec J_j^k \vec e_j^k =  - ((\vec e_i^k)^T m_{ij} (\vec J_j^k)^T \vec e_j^k)^T = - (\vec e_j^k)^T m_{ji} \vec J_j^k \vec e_i^k$, the requirement translates to
\begin{equation}
	\label{eq:04:0290}
    (\vec e_j^k)^T m_{ji} \vec J_i^k \vec e_i^k = (\vec e_j^k)^T m_{ji} \vec J_j^k \vec e_i^k \qquad \forall i \neq j,
\end{equation}
which is true if either of the following conditions holds.

\noindent
\hspace*{2ex} (C1)\quad  $m_{ij} = 0$ for all $i\neq j$,\\[1ex]
\hspace*{2ex} (C2)\quad $\vec J_i^k = \vec J_j^k = const.$ for all $i,j=1,\ldots,s$.\\[-2ex]

While (C1) is an orthogonality condition on the choice of the approximation basis in the \emph{discretization method}, the constant interconnection structure according to (C2) is a \emph{system property}\footnote{Darboux-Lie theorem states that every Poisson system can be transformed locally to a canonical Hamiltonian system with a constant skew-symmetric matrix (\cite{hairer2006geometric}, Section VII.3). Condition (C2) is then satisfied if the original PH systems is expressed in terms of these canonical local coordinates.}.

In both cases (C1) or (C2), the discrete energy balance \eqref{eq:04:0270} boils down to 
\begin{equation}
    \label{eq:04:0300}
    h (\vec e^k)^T \vec M \vec f^k + h(\vec e^k)^T \vec M \vec G^k \vec u^k = 0.
\end{equation}
The definition of a \emph{discrete-time output vector}
\begin{equation}
    \label{eq:04:0310}
   {\vec y^k} := (\vec G^k)^T \vec M \vec e^k
\end{equation}
yields (using $\vec M = \vec M^T$)
\begin{equation}
    \label{eq:04:0320}
    h (\vec M\vec e^k)^T \vec f^k + h (\vec y^k)^T \vec u^k = 0,
\end{equation}
which represents a \emph{structural balance equation} for the supplied energy in $I^k$ (or average supplied power, if divided by $h$). We are ready to define the discrete-time Dirac structure, which is based on the polynomial approximation of the power variables.

\begin{thm}[Discrete-time Dirac structure]
    \label{thm:04:0010}
    Given the $s$ collocation points $0 \leq c_i \leq 1$, $i = 1,\ldots,s$. The system of equations \eqref{eq:04:0220}, \eqref{eq:04:0310},  i.\,e.
    \begin{equation}
    \begin{split}
        \label{eq:04:0330}
        - \vec f^k &= \vec J^k  \vec e^k + \vec G^k \vec u^k\\
        \vec y^k &= (\vec G^k)^T \vec M \vec e^k,
    \end{split}
    \end{equation}    
    with discrete flow, effort and input vectors $\vec f^k$, $\vec e^k$, $\vec u^k$ according to \eqref{eq:04:0100}, \eqref{eq:04:0190}, \eqref{eq:04:0200}, the block matrices $\vec J^k = - (\vec J^k)^T$, $\vec G^k$ according to \eqref{eq:04:0210}, and the symmetric block matrix $\vec M = \vec M^T$ according to \eqref{eq:04:0260}, represents a \emph{discrete-time Dirac structure} on the time interval $I^k\!=\![(k-1)h, kh]$, if either condition (C1) or (C2) is satisfied.
\end{thm}

\begin{proof} Rewrite \eqref{eq:04:0330} in \emph{kernel} representation
\begin{equation}
    \label{eq:04:0340}
    \vec F \bmat{\vec f^k\\ \vec y ^k} + \vec E \bmat{\vec M \vec e^k\\ \vec u^k} = \vec 0
\end{equation}
with 
\begin{equation}
    \label{eq:04:0350}
    \vec F = \vec I_{s(n+m)}, \qquad 
    \vec E = \bmat{ \vec J^k  \vec M^{-1} & \vec G^k\\ - (\vec G^k)^T & \vec 0}.
\end{equation}
According to \cite{schaft2017l2}, Prop. 6.6.6, the subspace of discrete bond variables $(\vec f^k, \vec M \vec e^k)$ and $(\vec y^k, \vec u^k)$ on which \eqref{eq:04:0340} holds, is a Dirac structure, if (i) $\vec E \vec F^T + \vec F \vec E^T = \vec 0$ and (ii) the matrix $[ \vec F\; \vec E]$ has full row rank $s(n+m)$. The latter is obvious from $\vec F = \vec I_{s(n+m)}$. Condition (i) is satisfied if and only if $\vec J^k \vec M^{-1}$ or, equivalently, $\vec M \vec J^k$ is skew-symmetric. As shown above,  the latter holds if the continuous-time system has a constant interconnection matrix (C2) or the choice of collocation points guarantees (C1).
\end{proof}

\subsection{Discrete-time port-Hamiltonian system}
\label{subsec:04:030:060}
The discrete-time Dirac structure is now complemented by discrete-time dynamics and constitutive equations.

\begin{dfn}[Discrete-time PH system]
\label{def:04:0010}
Equations \eqref{eq:04:0330}, together with the \emph{$s$-stage discrete dynamics}
\begin{subequations}
\label{eq:04:0360}
\begin{align}
    \label{eq:04:0360a}
    \vec x_0^k &= \vec x_{s+1}^{k-1},\\
    \label{eq:04:0360b}
    \vec x_i^k &= \vec x_{0}^k - h \sum_{j=1}^{s} a_{ij} \vec f_j^k, \quad i=1,\ldots, s,\\
    \label{eq:04:0360c}
    \vec x_{s+1}^k &= \vec x_0^k - h \sum_{j=1}^s b_j \vec f_j^k,
\end{align}
\end{subequations}
with Runge-Kutta coefficients $a_{ij}$ and $b_j$ according to \eqref{eq:04:0140} and the \emph{discrete constitutive equations} 
\begin{equation}
    \label{eq:04:0370}
    \vec e_i^k = \left. \nabla H(\vec x) \right|_{\vec x = \vec x_i^k}, \qquad i=1,\ldots,s,
\end{equation}
define a \emph{discrete-time PH system}. Using \eqref{eq:04:0250} and \eqref{eq:04:0320}, the approximation of energy supplied to the storage elements on the sampling interval $I^k=$ \mbox{$[(k-1)h, kh]$} is given by
\begin{equation}
    \label{eq:04:0380}
    \Delta \tilde H^k =   h (\vec y^k)^T \vec u^k.
\end{equation}
\end{dfn}

\begin{rmk}
    Equations \eqref{eq:04:0360b} and  \eqref{eq:04:0360c} can be written in more compact form using the Kronecker product
    \begin{subequations}
		\label{eq:04:0390}
        \begin{align}
            \vec x_i^k &= \vec x_0^k - h (\vec a_i^T \otimes \vec I_n ) \vec f^k, \quad i=1,\ldots, s,\\
            \vec x_{s+1}^k &= \vec x_0^k - h (\vec b^T \otimes \vec I_n ) \vec f^k,
        \end{align}
    \end{subequations}
    with $\vec a_i^T = \bmat{a_{i1} & \ldots & a_{is}}$, $\vec b^T = \bmat{b_{1} & \ldots & b_{s}}$.
\end{rmk}

\begin{rmk}
    The \emph{structure} of the discrete-time PH models is independent of the sampling time $h$, which, however, determines the approximation quality.
\end{rmk}

\subsection{Stored energy approximation and discrete energy balance}
\label{subsec:04:030:070}
Equation \eqref{eq:04:0380} represents an approximation of the \emph{supplied energy} flow through the port $(\vec u(t), \vec y(t))$ based on the polynomial approximation of flows and efforts. On the other hand, the difference
\begin{equation}
	\label{eq:04:0391}
	\Delta \bar{H}^k = H(\vec x_{s+1}^k) - H(\vec x_{0}^k)
\end{equation}
expresses the increment of \emph{stored energy} based on a numerical integration scheme \eqref{eq:04:0360}, as opposed to the exact increment
\begin{equation}
	\label{eq:04:0392}
	\Delta {H}^k = H(\vec x(t_0^k + h)) - H(\vec x(t_0^k)).
\end{equation}

\begin{dfn}
	\label{dfn:04:0020}
	If under a given numerical integration scheme of order $p$, the increment of stored energy satisfies 
	\begin{equation}
		\label{eq:04:0393}
		\Delta \bar H^k = h (\vec y^k)^T  \vec u^k + o(h^p),
	\end{equation}
	we call \eqref{eq:04:0393} a \emph{discrete energy balance}, which is \emph{consistent} with the discretization scheme. If 
	\begin{equation}
		\label{eq:04:0394}
		\Delta \bar H^k = h (\vec y^k)^T  \vec u^k
	\end{equation}	
	we call the energy balance \emph{exact}. 
\end{dfn}

\begin{rmk}
	In Definition III.2 of \cite{celledoni2017energy}, the latter case is simply called ``discrete energy balance''. As \eqref{eq:04:0394} only holds under additional conditions, like constant structure and quadratic energy matrix, and for example under the implicit midpoint rule (see \cite{aoues2013canonical}, Section III.B and \cite{celledoni2017energy}, Section III.E), we add ``exact'' to distinguish this particular case of a discrete energy balance.
\end{rmk} 

To show the consistency of a discrete energy balance, we will show that the local approximation errors of both $\Delta \bar H^k$ and $\Delta \tilde H^k$, compared with the exact increment $\Delta H^k$ with $\vec x_0^k = \vec x(t_0^k)$, have order $o(h^p)$. To perform this analysis, we restrict ourselves to quadratic energies of the form 
\begin{equation}
	\label{eq:04:0395}
	H(\vec x) = \frac{1}{2} \vec x^T \vec Q \vec x, \qquad \vec Q = \vec Q^T > 0.
\end{equation}

\begin{thm}[Local error of stored energy]
	\label{thm:04:0020}
	For a quadratic energy \eqref{eq:04:0395}, the local energy error 
	\begin{equation}
		\label{eq:04:0398}
		\Delta \bar H^k - \Delta H^k, \qquad \vec x_0^k = \vec x(t_0^k),
	\end{equation}
	is consistent with the numerical integration scheme, i.\,e. it has order $o(h^p)$.
\end{thm}

\begin{proof} For an integration scheme of order $p$, the \emph{local} approximation error or consistency error (set $\vec x_0^k = \tilde{\vec x}(t_0^k) = \vec x(t_0^k)$) has order $o(h^p)$: $\| \tilde{\vec x}(t_0^k + h) - \vec x(t_0^k + h) \| \leq  C_1 h^{p+1}$, $C_1>0$. By the equivalence of norms, this holds accordingly for the energy norm $\|\vec x \|_Q := \sqrt{\frac{1}{2} \vec x^T \vec Q \vec x}$, with a different constant $C_2>0$: $\| \tilde{\vec x}(t_0^k + h) - \vec x(t_0^k + h) \|_{Q} \leq  C_2 h^{p+1}$. For the error in the supplied energy, the following estimate can be given ($t_{s+1}^k = t_0^k + h$), where the triangle inequality is used between third and fourth line ($\vec x^k_{s+1} = \tilde{\vec x}(t_{s+1}^k)$):
{\footnotesize
\begin{align}
	\nonumber
	\big| \Delta \bar H^k - \Delta H^k \big| 
		&= \left| \| \vec x_{s+1}^k \|_Q^2 - \| {\vec x}(t_{s+1}^k\hspace*{-0.25mm}) \|_Q^2 \right| \\
		\nonumber
		&\hspace*{-2cm} 
		= \left|  \left( \| \vec x_{s+1}^k \|_Q + \| {\vec x}(t_{s+1}^k) \|_Q \right) \cdot
		          \left(\| \vec x_{s+1}^k \|_Q - \| {\vec x}(t_{s+1}^k) \|_Q \right) \right| \\
		\nonumber
		&\hspace*{-2cm} 
		\leq 2 \max \left(  \|  \vec x_{s+1}^k \|_Q, \| {\vec x}(t_{s+1}^k) \|_Q  \right) \cdot
			\left|   \| \vec x_{s+1}^k \|_Q - \| {\vec x}(t_{s+1}^k) \|_Q  \right|\\
		\nonumber
		&\hspace*{-2cm} 
		\leq 2 \max \left( \| \vec x_{s+1}^k \|_Q, \| {\vec x}(t_{s+1}^k) \|_Q \right)  \cdot
			\left\| \vec x_{s+1}^k -  {\vec x}(t_{s+1}^k) \right\|_Q\\
		\label{eq:04:0430}
		&\hspace*{-2cm} 
		\leq C_3 h^{p+1}
\end{align}		
}
with $C_3 = 2 \max \left( \| \vec x^k_{s+1} \|_Q, \| {\vec x}(t_{s+1}^k) \|_Q \right) C_2$.
\end{proof}

In the following section, we will discuss Gauss-Legendre collocation and the Lobatto IIIA/IIIB pairs as examples for symplectic integration schemes. In the former case, we will prove that $\Delta \bar H^k = \Delta \tilde H^k$ holds for quadratic energies. In the latter case, we will show that $\Delta \tilde H^k$ approximates the energy increment $\Delta H^k$ with an error of identical order as $\Delta \bar H^k$, and consequently the discrete energy balance \eqref{eq:04:0393} holds.

\section{Examples and analysis of energy errors}
\label{sec:04:040}
The degrees of freedom to define a discrete-time PH system according to Definition \ref{def:04:0010} are (i) the set of collocation points $\{c_1, \ldots, c_s\}$ and (ii) the concrete coefficients $a_{ij}$ and weights $b_j$ of the RK integration scheme \eqref{eq:04:0360}. In this Section, we consider two classes of $s$-stage symplectic integration schemes: Gauss-Legendre collocation and the Lobatto IIIA/IIIB pairs for partitioned systems. In both cases, we analyze the error of the supplied energy approximation $\Delta \tilde H^k$ for quadratic energies and its relation to $\Delta \bar H^k$ in the light of Definition \ref{dfn:04:0020}.

\subsection{Gauss-Legendre collocation}
\label{subsec:04:040:010}
To define a discrete-time Dirac structure for cases with non-constant interconnection matrices, condition (C1) must be satisfied by the choice of collocation points. To have $m_{ij} = m_{ji} = 0$ for $i\neq j$, with $m_{ij}$ defined in \eqref{eq:04:0260}, the interpolation polynomials $\{ \ell_1(\tau), \ldots, \ell_s(\tau)\}$ must form a system of \emph{orthogonal} functions. This is the case, if we take the collocation points $c_1, \ldots, c_s$ as the zeros of the shifted Legendre polynomials\footnote{See \cite{hairer2006geometric}, Section II.1.3.}
\begin{equation}
	\label{eq:04:0450}
    \frac{d^s}{d\tau^s} \left( \tau^s (\tau-1)^s \right)
\end{equation}
in normalized time $\tau \in [0,1]$. With the resulting interpolation polynomials, the diagonal elements of $\vec M$,
\begin{equation}
    \label{eq:04:0460}
     m_{ii} = \int_0^1 \ell_i^2(\sigma)\, d\sigma = \int_0^1 \ell_i(\sigma)\, d\sigma = b_i,
\end{equation}
equal the weights $b_i$ in the Gauss-Legendre quadrature formula
\begin{equation}
	\label{eq:04:0470}
    \int_0^1 f(\sigma) \,d\sigma \approx \sum_{i=1}^s b_i f(c_i).
\end{equation}
This quadrature formula is \emph{exact} for polynomials $f(\tau)$ on $[0,1]$ up to order $2s-1$. With $2s$ the order of the quadrature formula, the approximation error of the integral 
\begin{equation}
	\label{eq:04:0480}
	 \int_{t_0}^{t_0+h}\!\!\!\!\!\!\!\! f(s)\, ds = h \! \int_0^1 \!\! f(t_0 + h \sigma) \, d\sigma \approx h \sum_{i=1}^s b_i f(t_0 + h c_i)
\end{equation}
is of order $\mathcal O(h^{2s+1})$. The coefficients $a_{ij}$ of the unique implicit Runge-Kutta (RK) method\footnote{See the Butcher tables in \cite{hairer2006geometric}, Section II.1.3.} of order $p=2s$ can be computed as given in \eqref{eq:04:0140}.

We now determine the conditions on the parameters of the integration scheme under which $\Delta \bar H^k = \Delta \tilde H^k$. Substituting \eqref{eq:04:0360c} in \eqref{eq:04:0391} for a quadratic energy, we have
\begin{multline}
	\label{eq:04:0500}
		\Delta \bar H^k 
		= - h (\vec x_0^k)^T \vec Q \sum_{j=1}^s b_j \vec f_j^k \\ 
		+ \frac{1}{2} h^2 (\sum_{j=1}^s b_j \vec f_j^k)^T \vec Q \sum_{j=1}^s b_j \vec f_j^k.
\end{multline}
On the other hand, with $b_j = m_{jj}$, $\vec e_j^k = \vec Q \vec x_j^k$ and \eqref{eq:04:0360c}, Equation \eqref{eq:04:0250} becomes
\begin{align}
	\nonumber
		\Delta \tilde H^k &= - h \sum_{j=1}^s (\vec e_j^k)^T m_{jj}  \vec f_j^k\\
	\label{eq:04:0510}
		&\hspace*{-0.75cm}= - h \sum_{j=1}^s (\vec x_0^k - h \sum_{l=1}^s a_{jl} \vec f_j^k )^T \vec Q b_j \vec f_j^k\\
	\nonumber
		&\hspace*{-0.75cm}= - h (\vec x_0^k)^T \vec Q \sum_{j=1}^s \! b_j \vec f_j^k 
		\!+\! h^2 \! \sum_{j=1}^s \! \Big( (\sum_{l=1}^s \! a_{jl} \vec f_l^k )^T \vec Q b_j \vec f_j^k \Big).
\end{align}
The first terms in \eqref{eq:04:0500} and \eqref{eq:04:0510} are identical. By matching the coefficients in front of $h^2 (\vec f_l^k)^T \vec Q \vec f_j^k$ in the second term, one obtains the conditions 
\begin{equation}
	\label{eq:04:0520}
	a_{ii} = \frac{1}{2} b_i, \quad a_{ij} b_i + a_{ji} b_j = b_i b_j
\end{equation}
for $i,j =1,\ldots, s$, under which $\Delta \bar H^k$ and $\Delta \tilde H^k$ coincide.

\begin{thm}
	\label{thm:04:0030}
	The $s$-stage Gauss-Legendre collocation methods, applied to PH systems with quadratic energy, are the only numerical integration schemes, which yield an exact discrete energy balance \eqref{eq:04:0394}.
\end{thm}

\begin{proof}
	Equation \eqref{eq:04:0520} is only satisfied for Gauss-Legendre collocation, see \cite{hairer2006geometric}, Section IV.2.1, Theorem 2.2 and the paragraph below the proof\footnote{The criterion \eqref{eq:04:0520} characterizes numerical integration methods that conserve quadratic invariants. ``[A]mong all collocation and discontinuous collocation methods [\ldots] only the Gauss methods satisfy this criterion [\ldots].''}. If   \eqref{eq:04:0520} is true, $\Delta \bar H^k = \Delta \tilde H^k$ holds, and \eqref{eq:04:0394} follows from \eqref{eq:04:0380}. 
\end{proof}

\begin{rmk}
	Gauss-Legendre collocation with $s=1$ leads to the implicit midpoint rule, which is shown in \cite{aoues2013canonical} to satisfy an exact discrete energy balance. Theorem \ref{thm:04:0030} shows that this is not the only choice for structure-preserving time discretization of PH systems with exact energy balance.
\end{rmk}

\subsection{Lobatto IIIA/IIIB pairs}
\label{subsec:04:040:020}
Partitioned collocation methods, such as Lobatto pairs are used for separable Hamiltonian systems. We consider in this study the linear PH system of simple mechanical type\footnote{$\vec Q = \vec K$ denotes the stiffness matrix, $\vec P = \vec M^{-1}$ the inverse mass matrix.} with $\vec q, \vec p \in \IR^n$, $\vec u \in \IR^m$,
\begin{equation}
	\label{eq:04:0530}
	\begin{split}
	\bmat{\dot {\vec q}(t)\\ \dot {\vec p}(t)}
	&=
	\bmat{\vec 0 & \vec I\\ - \vec I & \vec 0}
	\bmat{\vec Q \vec q(t)\\ \vec P \vec p(t)}
	+
	\bmat{\vec 0\\ \vec G}
	\vec u(t)\\
	\vec y(t) &= \bmat{\vec 0 & \vec G^T} 
	\bmat{\vec Q \vec q(t)\\ \vec P \vec p(t)}.	
	\end{split}
\end{equation}
The discrete-time structure equations can be expressed as
\begin{equation}
	\label{eq:04:0540}
	\bmat{- \vec f_q^k\\ - \vec f_p^k}
	=
	\bmat{\vec 0 & \vec I_{sn}\\ - \vec I_{sn} & \vec 0}
	\bmat{\vec e_q^k\\ \vec e_p^k}
	+
	\bmat{ \vec 0\\ \vec G^k }
	\vec u^k.
\end{equation}
The elements of the effort and input vectors $\vec e_q^k, \vec e_p^k \in \IR^{sn}$, $\vec u^k \in \IR^{sm}$ are 
\begin{equation}
	\label{eq:04:0550}
	\vec e_{q,i}^k = \vec Q \vec q_i^k, \quad\!
	\vec e_{p,i}^k = \vec P \vec p_i^k, \quad\!
	\vec u_i^k = \vec u(t_0^k + c_i h).
\end{equation}
The partitioned integration scheme, which consists of two RK methods (coefficients $a_{ij}$, $\hat a_{ij}$, $b_j = \hat b_j$, $c_i = \hat c_i$, $i,j=1,\ldots, s$), each applied to one set of differential equations, can be written ($i=1,\ldots, s$)
\begin{subequations}
\label{eq:04:0560}
\begin{align}
	\label{eq:04:0560a}
	\vec q_0^k & \!=\! \vec q_{s+1}^{k-1}, & 
		\vec p_0^k & \!=\! \vec p_{s+1}^{k-1}, \\
	\label{eq:04:0560b}
	\vec q_i^k & \!=\! \vec q_0^k \!-\! h \sum_{j=1}^s a_{ij} \vec f_{q,j}^k, & 
		\vec p_i^k & \!=\! \vec p_0^k \!-\! h \sum_{j=1}^s \hat a_{ij} \vec f_{p,j}^k,\\		
	\label{eq:04:0560c}
	\vec q_{s+1}^k & \!=\! \vec q_0^k \!-\! h \sum_{j=1}^s b_{j} \vec f_{q,j}^k, & 
		\vec p_{s+1}^k & = \vec p_0^k \!-\! h \sum_{j=1}^s \hat b_{j} \vec f_{p,j}^k,
\end{align}
\end{subequations}
with $\vec f_{q,j}^k = - \vec e_{p,j}^k$, $\vec f_{p,j}^k = \vec e_{q,j}^k - \vec G^k \vec u_j^k$, $j=1,\ldots,s$.

In contrast to Gauss-Legendre collocation, the expression for the increment of a quadratic Hamiltonian per sampling interval
\begin{multline}
	\label{eq:04:0570}
		\Delta \bar H^k = \frac{1}{2} (\vec q_{s+1}^k)^T \vec Q \vec q_{s+1}^k - \frac{1}{2} \vec (q_0^k)^T \vec Q \vec q_0^k\\
		+ \frac{1}{2} (\vec p_{s+1}^k)^T \vec P \vec p_{s+1}^k - \frac{1}{2} (\vec p_0^k)^T \vec P \vec p_0^k
\end{multline}
does \emph{not} coincide with the approximate supplied energy
\begin{equation}
	\label{eq:04:0580}
	\begin{split}
		\Delta \tilde H^k &= - \int_{t_0^k}^{t_0^k +h} \tilde{\vec e}_q^T(s) \tilde{\vec f}_q(s) + \tilde{\vec e}_p^T(s) \tilde{\vec f}_p(s) \, ds\\
		&= - h (\vec e_q^k)^T \vec M \vec f_q^k - h (\vec e_p^k)^T \vec M \vec f_p^k,
	\end{split}
\end{equation}
where $\vec M$ is given by \eqref{eq:04:0260}.

We restrict our attention to the 3-stage Lobatto pair\footnote{See the Butcher tables in \cite{hairer2006geometric}, Section II.1.4.} with collocation points $c_1 = 0$, $c_2 = \frac{1}{2}$, $c_3 = 1$, in which case the matrix $\vec M$ is
\begin{equation}
	\label{eq:04:0585}
	\vec M = \bmat{\frac{2}{15} & \frac{1}{15} & -\frac{1}{30}\\ \frac{1}{15} & \frac{8}{15} & \frac{1}{15}\\ -\frac{1}{30} & \frac{1}{15} & \frac{2}{15}} \otimes \vec I_n.
\end{equation}

\begin{thm}
	\label{thm:04:0040}
	For the $3$-stage Lobatto pair, applied to the partitioned linear PH system \eqref{eq:04:0530}, the error between $\Delta \tilde H^k$ and $\Delta \bar H^k$, and consequently the error between $\Delta \tilde H^k$ and $\Delta H^k$ has order $o(h^{2s-2}) = o(h^4)$.
\end{thm}

\begin{proof}
	First notice that the local energy error $\Delta \bar H^k - \Delta H^k$ is of order $o(h^p)$ with $p=2s-2$ the order of the Lobatto pair, see Theorem \ref{thm:04:0010}. To prove that the error $\Delta \tilde H^k - \Delta \bar H^k$ has the same order, the expressions in \eqref{eq:04:0570} and \eqref{eq:04:0580} are subtracted, under substitution of the efforts and states in the $i$-th stage according to \eqref{eq:04:0550} and \eqref{eq:04:0560b}. We replace the terms $\vec f_{q,1}^k$, $\vec f_{q,3}^k$ and $\vec f_{p,1}^k$, $\vec f_{p,3}^k$ by their Taylor expansions
	\begin{equation}
	\label{eq:04:0590}
	\begin{split}
		\vec f_{q,1}^k \!&=\! \vec f_{q,2}^k \!-\! \dot{\vec f}_{q,2}^k \frac{h}{2} \!+\! \vec r_{1} h^2, \quad
		\vec f_{q,3}^k \!=\! \vec f_{q,2}^k \!+\! \dot{\vec f}_{q,2}^k \frac{h}{2} \!+\! \vec r_{2} h^2\\
		\vec f_{p,1}^k \!&=\! \vec f_{p,2}^k \!-\! \dot{\vec f}_{p,2}^k \frac{h}{2} \!+\! \vec r_{3} h^2, \quad
		\vec f_{p,3}^k \!=\! \vec f_{p,2}^k \!+\! \dot{\vec f}_{p,2}^k \frac{h}{2} \!+\! \vec r_{4} h^2,
	\end{split}	
	\end{equation}
	where $\dot{\vec f}_{q,2}^k = \frac{d}{dt}\tilde{\vec f}_{q}(t_2^k)$ and $\dot{\vec f}_{p,2}^k = \frac{d}{dt}\tilde{\vec f}_{p}(t_2^k)$ are the time derivatives of the polynomial flow approximations $\tilde{\vec f}_{q}$ and $\tilde{\vec f}_{p}$ in $t_2^k = t_0^k + \frac{h}{2}$. $\vec r_{1}, \ldots, \vec r_4$ are residual terms. The result is 
	\begin{multline}
	\label{eq:04:0610}
		\Delta \bar H^k - \Delta \tilde H^k 
		= F_1(\vec r_1, \ldots, \vec r_4, \dot{\vec f}_{q,2}^k, \dot{\vec f}_{p,2}^k)
		h^5  \\
		+ F_1(\vec r_1, \ldots, \vec r_4) h^6,
	\end{multline}
	where $F_1$ and $F_2$ are functions in the given arguments. This, together with the order of the error $\Delta \bar H^k - \Delta H^k$, proves the claim.
\end{proof}

As a consequence of Theorem \ref{thm:04:0040}, the application of the $3$-stage Lobatto pair to the partitioned PH system \eqref{eq:04:0530} defines a discrete-time PH system, whose discrete energy balance is consistent with the order $p = 2s-2=4$ of the numerical scheme.

Theorem \ref{thm:04:0040} shows exemplarily at the case $s=3$ how to prove the identical consistency order of both local energy approximation errors. The numerical experiments in the following section give evidence that the corresponding order statement also holds for the  $4$-stage Lobatto pair.

\section{Numerical experiments}
\label{sec:04:050}
We illustrate the quantitative statements concerning the accuracy of the energy approximations by the numerical simulation of a linear oscillator, whose solutions can be computed, and which therefore serves as a benchmark example. First, the conservative case, which has been considered throughout the paper, is studied. The accumulated errors of energy supplied through the port $(u(t), y(t))$ and stored energy are determined and illustrated for both considered families of integration schemes. In a second part, the control port is closed by constant feedback, which injects damping to the system. We show that the accuracy order of the energy approximation is maintained in the lossy case.

The considered state PH model of the lossless oscillator is given by the explicit representation of the underlying Dirac structure 
\begin{subequations}
\label{eq:04:0670}
\begin{align}
	- \vec f(t) &= \vec J \vec e(t) + \vec g u(t)\\
	y(t) &= \vec g^T \vec e(t)
\end{align}
\end{subequations}
with flow and effort vectors $\vec f, \vec e \in \IR^2$, in- and output $u, y \in \IR$.
\begin{equation}
\label{eq:04:0680}
	\vec J = \bmat{0 & 1\\ - 1 & 0}  \qquad \text{and} \qquad
	\vec g = \bmat{0\\1}
\end{equation}
denote the interconnection matrix and the input vector. The dynamics equation is $\dot{\vec x}(t) = - \vec f(t)$ with $\vec x\in \IR^2$ the state vector. The linear constitutive equations $\vec e(t) = \vec Q \vec x(t)$ are derived from the quadratic Hamiltonian $H(\vec x) = \frac{1}{2} \vec x^T \vec Q \vec x$ with  $\vec Q = \vec I$. For the lossy case, the extended output feedback
\begin{equation}
\label{eq:04:0690}
	u(t) = - r y(t) + v(t), \qquad r > 0,
\end{equation}
with new input $v(t)$, generates the damped system's state differential equation
\begin{equation}
	\label{eq:04:0700}
	\dot{\vec x}(t) = (\vec J - \vec R) \vec Q \vec x(t) + \vec g v(t), \quad \vec R = \bmat{0 & 0\\ 0 & r}.
\end{equation}

The structure equations \eqref{eq:04:0670} are discretized using collocation as described in the previous sections. First, Gauss-Legendre collocation with $s=1,2,3$ stages is used. Then, the partitioned representation of \eqref{eq:04:0670} is considered for the discretization of the Dirac structure with $3$- and $4$-stage Lobatto pairs. Discrete-time dynamics and constitutive equations are discretized according to Definition \ref{def:04:0010}, again considering the partitioned version of the state space model for the Lobatto pairs.

\subsection{Energy supply and storage in the lossless case}
\label{subsec:04:050:010}
Starting from an initial state $\vec x(0) = \bmat{q(0) & p(0)}^T = \bmat{0 & -1}^T$, the undamped system is excited by a pulse-shaped input
\begin{equation}
	\label{eq:04:0710}
	u(t) = \begin{cases}
		0, & t < 8\\
		\sin^2( \frac{t - 8}{10 - 8} \pi ), & 8 \leq t \leq 10 \\
		0, & t > 10.
	\end{cases}
\end{equation}

\begin{figure}[t]
	\centering
	\includegraphics[scale=0.8]{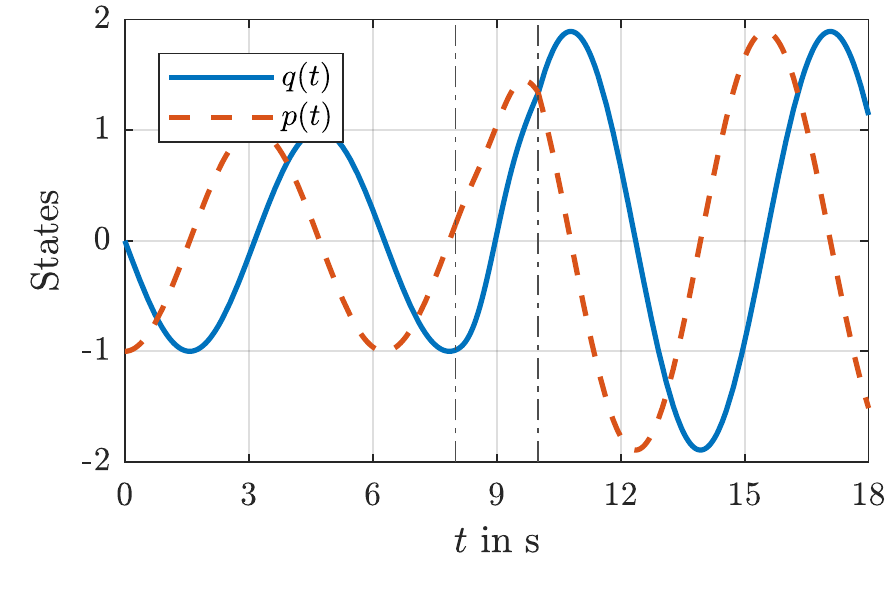}
	\qquad
	\includegraphics[scale=0.8]{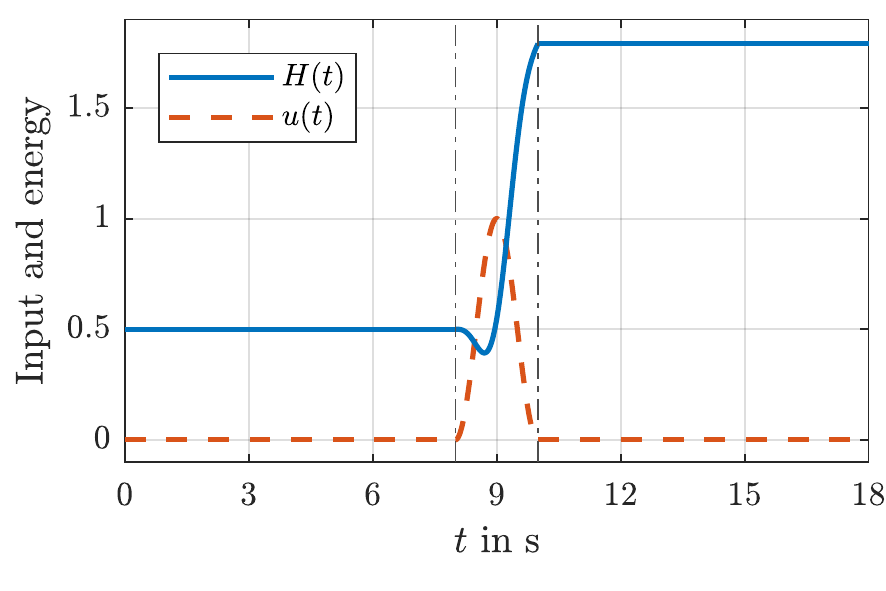}
	\caption{Evolution of states (top), input and energy (bottom) for the forced undamped oscillator}
	\label{fig:04:0020}
\end{figure}

\begin{figure}[t]
	\centering
	\includegraphics[scale=0.8]{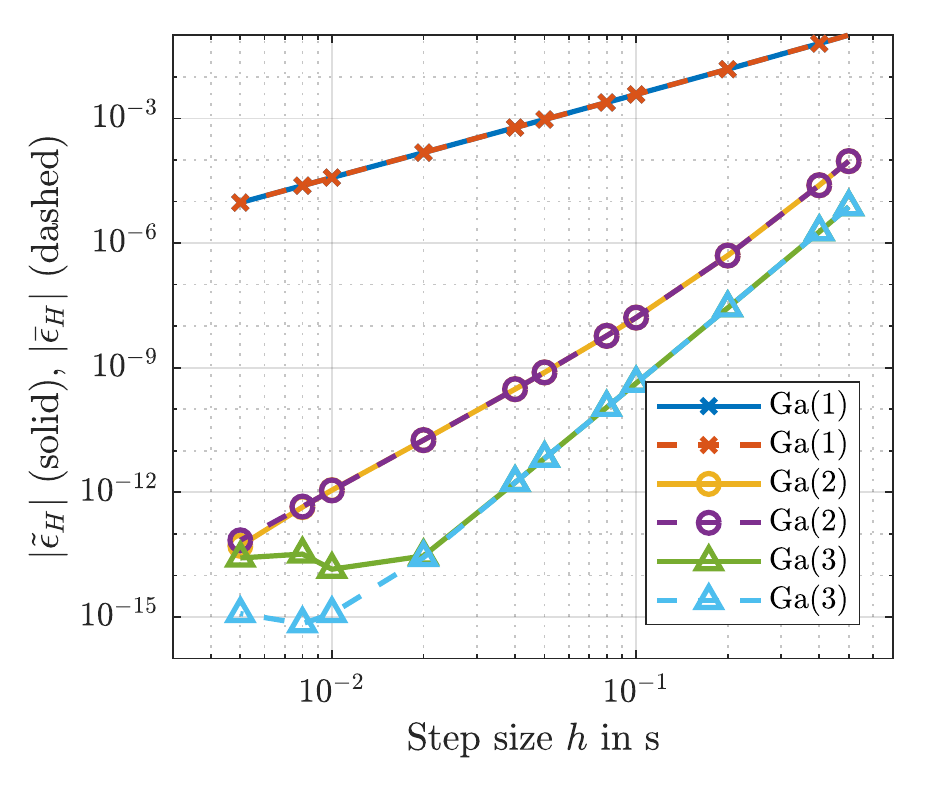}
	\qquad
	\includegraphics[scale=0.8]{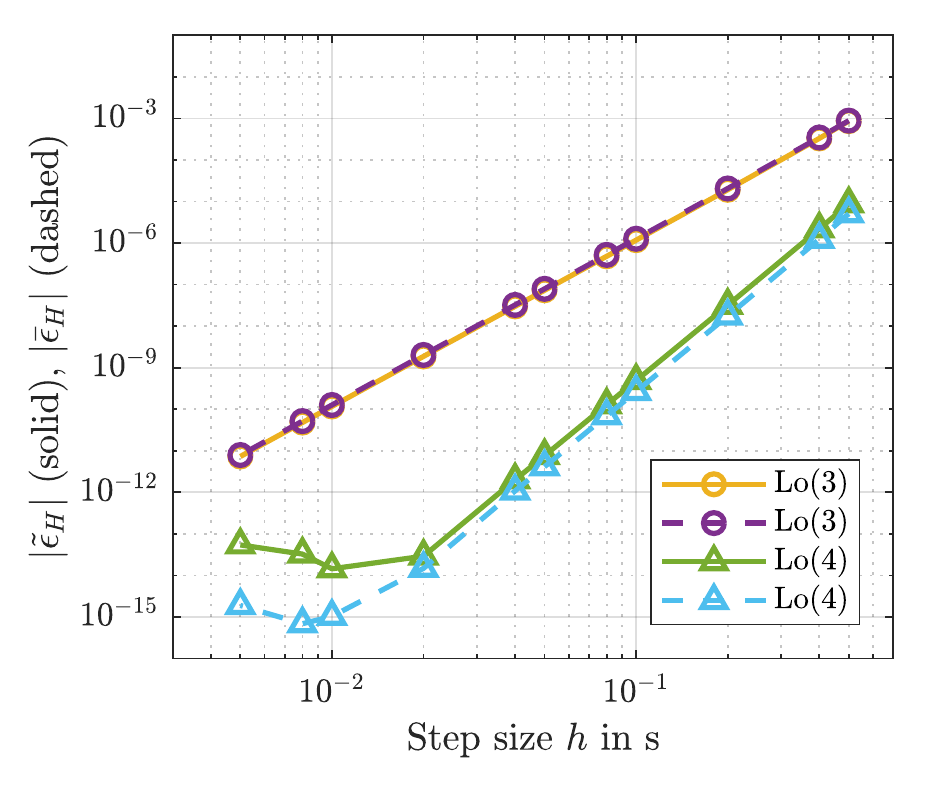}
	\caption{Errors $\tilde\epsilon_H$ and $\bar\epsilon_H$ of supplied and stored energy for Gauss-Legendre methods, $s=1,2,3$ (top) and Lobatto pairs, $s=3,4$ (bottom)}
	\label{fig:04:0030}
\end{figure}

Figure \ref{fig:04:0020} shows the exact evolutions of states, input and the quadratic energy for this test case on the interval $[0, T_{\mathrm{end}}] = [0, 18]$. Figure \ref{fig:04:0030} depicts the magnitude of relative errors
\begin{equation}
	\label{eq:04:0720}
	\tilde \epsilon_H := \frac{\Delta \tilde H_{\mathrm{tot}} - \Delta H_{\mathrm{tot}}}{\Delta H_{\mathrm{tot}}}, \qquad
	\bar \epsilon_H := \frac{\Delta \bar H_{\mathrm{tot}} - \Delta H_{\mathrm{tot}}}{\Delta H_{\mathrm{tot}}}
\end{equation}
of \emph{total} supplied and stored energy over the range of step sizes $h \in [0.005, 0.5]$. 
$\Delta H_{\mathrm{tot}} = \sum_{k=1}^N \Delta H^k$,
$\Delta \tilde H_{\mathrm{tot}} = \sum_{k=1}^N \Delta \tilde H^k$ and
$\Delta \bar H_{\mathrm{tot}} = \sum_{k=1}^N \Delta \bar H^k$
denote the total increment of energy and its approximations on $[0, T_{\mathrm{end}}]$ with $N = T_{\mathrm{end}}/h$ the number of sampling intervals. With $\Delta \bar H^k = \Delta H^k +  c^k h^{p+1}$, where $p$ is the order of the integration scheme, the absolute value of $\bar \epsilon_H$ can be bounded as follows:
\begin{equation}
\label{eq:04:0735}
	| \bar \epsilon_H | =
		\frac{|  \sum_{k=1}^N  c^k h^{p+1}|}{ |\sum_{k=1}^N \Delta H^k| }
	\leq \frac{\max_k |c^k|}{|P_{\mathrm{av}}|} h^p.
\end{equation}
$P_{\mathrm{av}} = \frac{ \Delta H_{\mathrm{tot}} }{N h}$ denotes the average transferred power, and the same estimation of order can be performed for $\tilde \epsilon_H$.

\begin{figure}[t]
	\centering
	\includegraphics[scale=0.8]{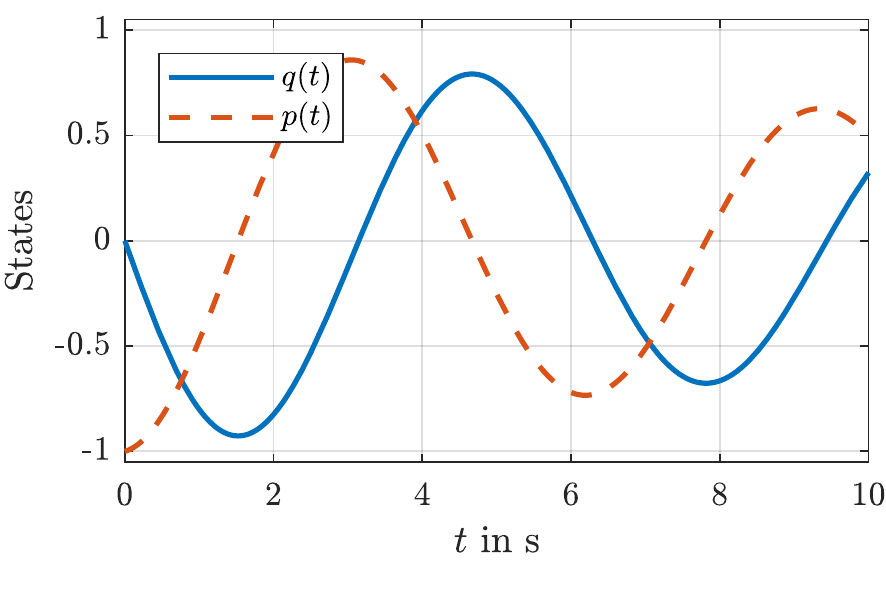}
	\qquad
	\includegraphics[scale=0.8]{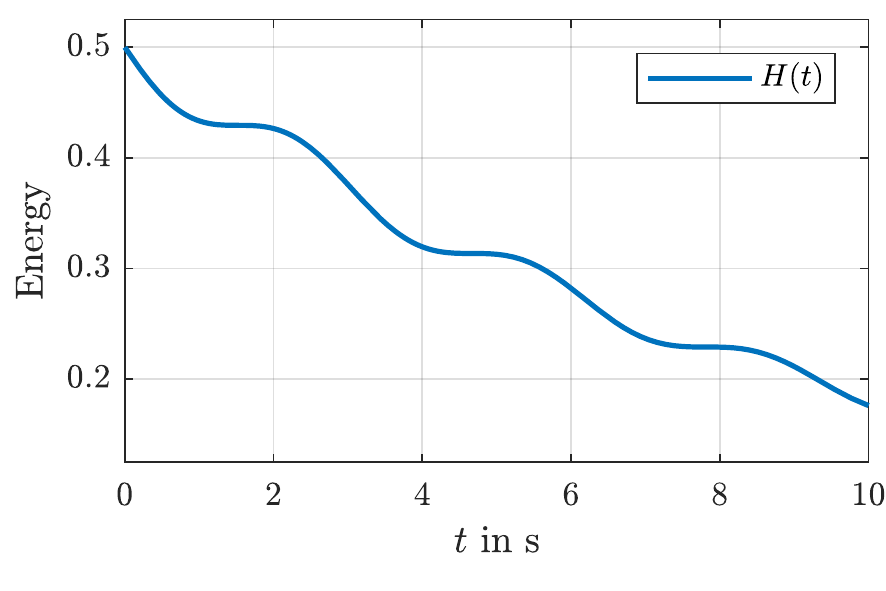}
	\caption{Evolution of the states (top) and the energy (bottom) for the damped oscillator}
	\label{fig:04:0040}
\end{figure}

\begin{figure}[t]
	\centering
	\includegraphics[scale=0.8]{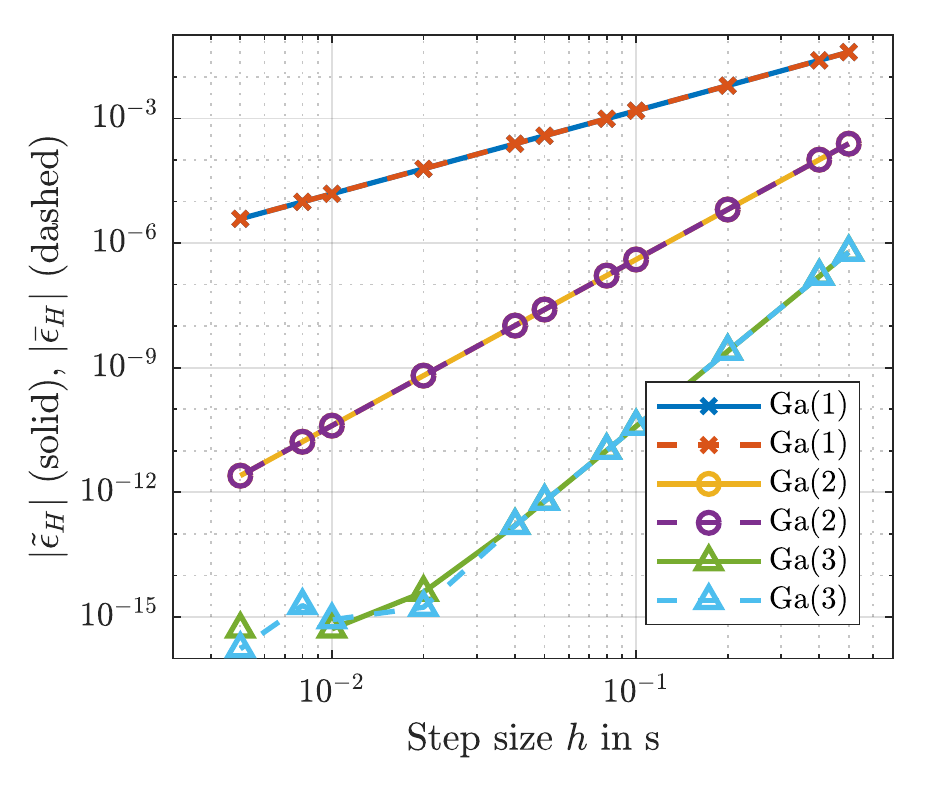}
	\qquad
	\includegraphics[scale=0.8]{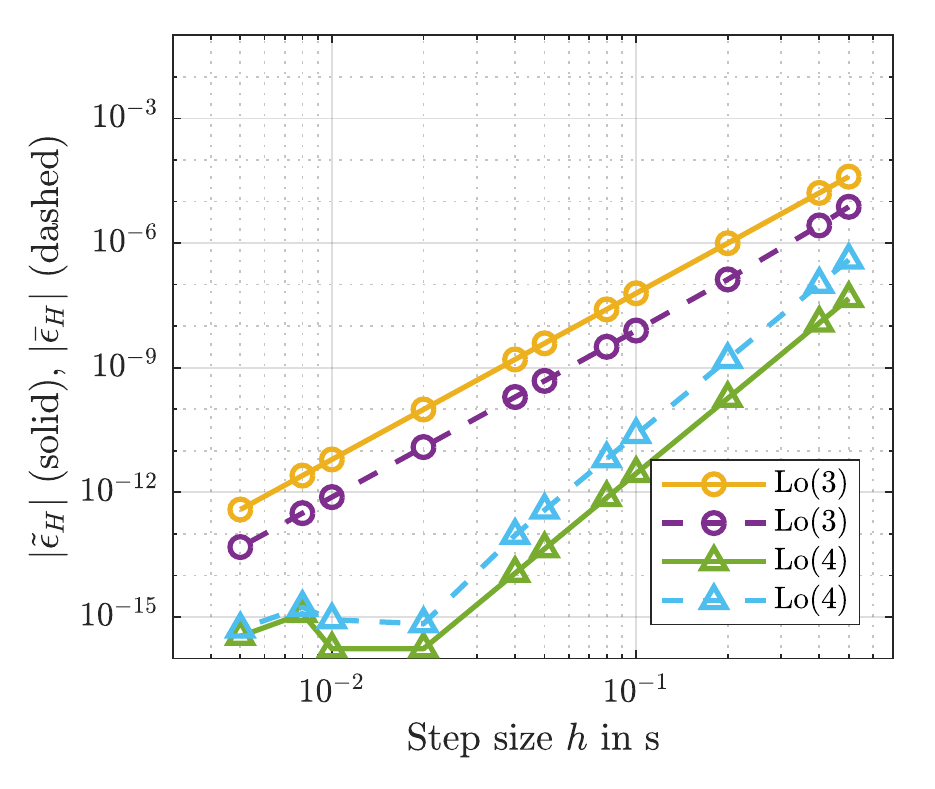}
	\caption{Errors $\tilde\epsilon_H$ and $\bar\epsilon_H$ of energy flow over the dissipative port and energy loss in the storage elements for Gauss-Legendre methods, $s=1,2,3$ (top) and the Lobatto pair, $s=3$ (bottom)}
	\label{fig:04:0050}
\end{figure}

The first diagram in Figure \ref{fig:04:0030} nicely shows the orders $2$, $4$ and $6$ of the Gauss-Legendre methods as well as the fact that both approximations $\Delta \tilde H^k$ and $\Delta \bar H^k$ of supplied and stored energy coincide. (The effect of rounding errors becomes visible at low step sizes in the curve for $s=3$.) The slopes in the second diagram confirm the orders $4$ and $6$ of the three-stage Lobatto pair. Although hardly recognizable, the two curves for $\tilde \epsilon_H$ and $\bar \epsilon_H$ do not match, which is accordance to the computations, see Eq. \eqref{eq:04:0610}.

\subsection{Approximation of dissipated energy}
\label{subsec:04:050:020}

The damped oscillator represents the most basic power-conserving interconnection of a port-Hamiltonian system (the undamped oscillator) with another system (a purely resistive element). This is nicely seen if \eqref{eq:04:0670} is combined with the damping injection feedback \eqref{eq:04:0690}. The differential energy balance in this damped case reads
\begin{equation}
	\label{eq:04:0750}
	\dot H = - r y^2 + y v \leq y v,
\end{equation}
which is the balance of power to the energy storage elements, supplied power and dissipated power. In the discrete-time setting, the damping injection output feedback becomes simply
\begin{equation}
	\label{eq:04:0760}
	\vec u^k = - r \vec y^k + \vec v^k
\end{equation} 
with discrete input $\vec u^k$ on the $k$-th sampling interval as defined in \eqref{eq:04:0200} (accordingly for $\vec v^k$) and the discrete output $\vec y^k$ according to \eqref{eq:04:0310}. By substitution in \eqref{eq:04:0380}, the approximate energy loss per time step becomes
\begin{equation}
	\label{eq:04:0770}
	\Delta \tilde H^k \!=\! - r h (\vec y^k)^T \vec y^k \!+\! h (\vec y^k)^T \vec v^k \leq  h (\vec y^k)^T \vec v^k.
\end{equation}

Figure \ref{fig:04:0040} shows the solution of \eqref{eq:04:0700} with damping parameter $r=0.1$ for $v\equiv 0$ and an initial value $\vec x(0) = \bmat{ p(0) & q(0) }^T = \bmat{ 0 & - 1}^T$ on the time interval $[0,10]$, as well as the monotonous decrease in energy. Figure \ref{fig:04:0050} depicts the magnitude of the relative errors of total dissipated energy according to \eqref{eq:04:0720}. As in the lossless case, the error plots confirm that the dissipated power is discretized consistently with the order of the underlying geometric numerical integration scheme. This time, the discrepancy between the numerical energy increments $\Delta \tilde H^k$ and $\Delta \bar H^k$ for the Lobatto pairs, which is of order $\mathcal O(h^p)$, is clearly visible in the right diagram.

\section{Conclusions}
\label{sec:04:060}

We presented a new definition of discrete-time PH systems, which is based on the approximation of the structure equations and the energy balance of explicit PH systems using the collocation method. By defining a discrete-time Dirac structure, discrete-time constitutive equations and using appropriate geometric numerical integration schemes, the separation between structure, constitutive laws and dynamics, which is a central feature of PH systems, is maintained. The presented work extends in a very natural way the notion of geometric/symplectic integration of autonomous Hamiltonian systems to the \emph{open} case (i.\,e. with power flow over the system boundary) of PH systems.

A focus has been set on proving consistency of the two different numerical energy increments that appear in the context of this definition. The family of implicit Gauss-Legendre schemes -- applied to linear PH systems -- is the only one for which the approximations of supplied and stored energy match, which leads to an \emph{exact} discrete energy balance of the discrete-time PH approximation. For Lobatto IIIA/IIIB pairs, applied to a linear PH system of partitioned mechanical structure, the discrete energy balance is not exact, but the energy error is consistent with the numerical integration scheme. The theoretical findings have been illustrated by numerical experiments with the simplest test case of a linear oscillator: The evolution of total energy, which is  (i) supplied by an external input or (ii) dissipated based on output damping injection, is approximated up to the order of the underlying integration scheme.

The presented definition of discrete-time PH systems can be exploited in the simulation and numerical analysis of large scale networks. The consistent approximation of energy flows between subsystems and the quantification of their error give important insight that helps to keep track of the quality of simulation results. In the context of network simulation, the numerical approximation of PH DAE systems \cite{schaft2013port-dae}, \cite{beattie2017port} is of particular interest. Combined with structure-preserving spatial discretization, see e.\,g. \cite{kotyczka2018weak}, the presented approach contributes to the full discretization of distributed-parameter PH systems. Moreover, the presented work gives rise to reconsider the Control by Interconnection approach, see e.\,g. \cite{ortega2008control}, for the stabilization of PH systems in discrete time. Finally, more general choices for the time discretization of effort and flow variables are conceivable, which would lead to interesting implicit representations of Dirac structures and, consequently, to implicit discrete dynamics.

\bibliographystyle{elsarticle-num}
\bibliography{koty-lefe_discrete-time-phs_v1}

\section*{Acknowledgement}
The work was supported by Deutsche Forschungsgemeinschaft (DFG), project KO 4750/1-1, and DFG-ANR (Agence Nationale de la Recherche), project INFIDHEM, ID ANR-16-CE92-0028.

\end{document}